\documentclass[a4,10pt]{article}

\usepackage{wrapfig}

\usepackage{lineno}

\usepackage{amsmath}
\usepackage{amsthm}
\usepackage{amssymb}
\usepackage{enumerate}
\usepackage[pdftex]{graphicx}
\usepackage{float}

\usepackage[top=36truemm,bottom=35truemm,left=32truemm,right=32truemm]{geometry}

\usepackage{titlesec}
\titleformat*{\section}{\large\bfseries}

\theoremstyle{definition}
\newtheorem{theorem}{Theorem}[section]

\newtheorem{lemma}[theorem]{Lemma}

\newtheorem{proposition}[theorem]{Proposition}
\newtheorem{remark}[theorem]{Remark}
\numberwithin{equation}{section}

\allowdisplaybreaks

\begin{document}
  \title{\Large{Asymptotic evaluations of generalized Bessel function of order zero related to the $p$-circle lattice point problem}}
  \author{Masaya Kitajima}
  \date{}
  \maketitle
  \begin{abstract}
  Let $p$ and $r$ be positive real numbers. Then, we consider the lattice point problem of the closed curve $p$-circle $\{x\in\mathbb{R}^{2}|\ |x_{1}|^{p}+|x_{2}|^{p}=r^{p}\}$ which is a generalization of the circle ($p=2$). Following the harmonic analytic approach of S. Kuratsubo and E. Nakai for the case of a circle, we need to investigate properties of appropriately generalized Bessel functions for $p$ in order to tackle the problem. Thus, in this paper, we derive asymptotic evaluations of the generalized Bessel function of order zero, such as uniformly asymptotic estimates on compact sets on quadrants of $\mathbb{R}^{2}$ for the cases $0<p<1$ or $p=2$, and, as stronger results, uniformly asymptotic estimates on $\mathbb{R}^{2}$ for the cases $p$ such that $\frac{2}{p}$ are the natural numbers other than $2$.
  \end{abstract}
  
  \section{Introduction and main results}
  \hspace{13pt}For positive real numbers $p$ and $r$, we consider the lattice point problem of the $p$-circle (this closed curve called Lam\'{e}'s curve or superellipse) $\{x\in\mathbb{R}^{2}|\ |x_{1}|^{p}+|x_{2}|^{p}=r^{p}\}$, which is a generalization of the circle. That is, the problem is to find a value $\alpha_{p}$ such that $P_{p}(r)=\mathcal{O}(r^{\alpha_{p}})$ and $P_{p}(r)=\Omega(r^{\alpha_{p}})$ hold, where $P_{p}$ is the area error term between the $p$-circle and a mosaic-like approximated figure.
  \begin{equation}
  P_{p}(r):=R_{p}(r)-\frac{2}{p}\frac{\Gamma^{2}(\frac{1}{p})}{\Gamma(\frac{2}{p})}r^{2}.
  \end{equation}
  Note that, $R_{p}$ is the number of lattice points in the $p$-circle, and $\Gamma$ is the gamma function, and, for the functions $f$ and $g$, $f(t)=\mathcal{O}(g(t))$ and $f(t)=\Omega(g(t))$ respectively mean $\limsup_{t\to a}$$|\frac{f(t)}{g(t)}|<+\infty$, $\liminf_{t\to a}$$|\frac{f(t)}{g(t)}|>0$ ($a$ is $\infty$ or constant). \par\vspace{5pt}
  In particular, in the case $p=2$ (that is, when the closed curve is a circle), it is called Gauss's circle problem\cite{Gauss}. In 1917, G.H. Hardy\cite{Hardy-1917} had conjectured the infimum of $\mathcal{O}$-estimates as
  \begin{equation}\label{Hardy}
    P_{2}(r)=\mathcal{O}(r^{\frac{1}{2}+\varepsilon}),\quad \neq\mathcal{O}(r^{\frac{1}{2}})
    \quad\text{as }r\to\infty
  \end{equation}
  for any $\varepsilon>0$ small enough, and M.N. Huxley\cite{Huxley-2003} showed that this evaluation holds for any $\varepsilon$ approximately greater than $\frac{27}{208}(=0.1298\cdots)$ in 2003. 
  
  \begin{figure}[t]
    \centering
    \includegraphics[width=0.8\linewidth]{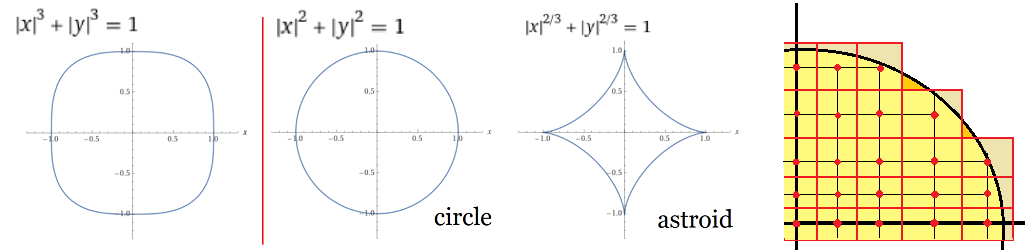}
    \label{fig}
    \caption{Examples of the $p$-circle and the approximation by unit squares.}
  \end{figure}
  
  \par\vspace{5pt}
  On the other hand, in the cases $p>2$, the following important theorem by E. Kr\"{a}tzel is given by the representation (\cite{Kratzel}, (3.57))
  \begin{equation}
    P_{p}(r)=\Psi(r;p)+\Delta(r;p)
  \end{equation}
  decomposed by the second main term $\Psi(r;p)$ (\cite{Kratzel}, (3.55)), which is represented as the series expansion consisting of the generalized Bessel functions (\cite{Kratzel}, Definition 3.3). It can be seen that the application of the result in 1993 by G. Kuba \cite{Kuba} (the $\mathcal{O}$-estimate of the remainder term $\Delta(r;p)$) have solved the problem in the cases $p>\frac{73}{27}$.   
  \begin{theorem}[\itshape{\cite{Kratzel}, Theorem 3.17 A}]\label{upper}
    \itshape{Let $p>2$. If $\alpha_{p}<1-\frac{1}{p}$ such that $\Delta(r;p)=\mathcal{O}(r^{\alpha_{p}})$ exists, then $P_{p}(r)=\mathcal{O}(r^{1-\frac{1}{p}}),\Omega(r ^{1-\frac{1}{p}})$ holds.}
  \end{theorem}
  \par\vspace{8pt}
  Now, we consider the remaining cases $0<p<2$. The conventional method in the cases $p\geq2$ deals with the second-order derivative $x_{2}''(x_{1})$ for $x_{2}>0$ with $r$ fixed. In the cases $0<p<2$, however, we need to approach the problem from another method, since singularities appear in this derivative. \par\vspace{7pt}
  As a starting point, we focus on the harmonic analytic approach to the problem for the case $p=2$ by S. Kuratsubo and E. Nakai\cite{Kuratsubo-2022}, in particular their very important equality (\cite{Kuratsubo-2022}, (2.6))
  \begin{equation}\label{D-J}
      D_{\beta}(s:x)-\mathcal{D}_{\beta}(s:x)=s^{\beta+1}2^{\beta+1}\pi\sum_{ n\in\mathbb{Z}^{2}\setminus\{0\}}\frac{J_{\beta+1}(2\pi\sqrt{s}|x-n|)}{(2\pi\sqrt{s}|x-n|)^{\beta+1}}\qquad\text{if }\beta>\frac{1}{2}.
  \end{equation}
  Note that, for $\beta>-1,\ s>0,\ x\in\mathbb{R}^{2}$, the functions are defined as 
  \begin{equation}\label{D}
    D_{\beta}(s:x):=\frac{1}{\Gamma(\beta+1)}\sum_{|m|^{2}<s}(s-|m|^{2})^{\beta}
    e^{2\pi ix\cdot m},\quad 
    \mathcal{D}_{\beta}(s:x):=\frac{1}{\Gamma(\beta+1)}\int_{|\xi|^{2}<s}
    (s-|\xi|^{2})^{\beta}e^{2\pi ix\cdot \xi}d\xi,
  \end{equation}
  and in particular, $D_{0}(r^{2}:0)-\mathcal{D}_{0}(r^{2}:0)=P_{2}(r)$ holds. Kuratsubo and Nakai gave \textit{a harmonic analytic claim equivalent to the Hardy's conjecture (\ref{Hardy})} by using this series representation (\ref{D-J}) via various properties of the Bessel functions (see Theorem 7.1 in \cite{Kuratsubo-2022}). On the other hand, as a byproduct of the process, the following evaluation formula is obtained. 
   \begin{proposition}[\itshape{\cite{Kuratsubo-2022}; Lemma 5.1(Special cases)}]\label{prop2022}\vspace{-3pt}
    \itshape{Under the above definitions of the functions,}
    \begin{equation*}\vspace{-2pt}
      D_{0}(s:x)-\mathcal{D}_{0}(s:x)=\mathcal{O}(s^{\frac{1}{3}})
      \quad\text{as }s\to\infty
    \end{equation*}
    \itshape{uniformly with respect to $x\in(-\frac{1}{2},\frac{1}{2}]^{2}$. In particular, 
    $P_{2}(r)=\mathcal{O}(r^{\frac{2}{3}})$ holds.}
    \end{proposition}
  \vspace{5pt}
  
  Therefore, as the future goal, following Kuratsubo and Nakai's method, we have decided to improve $\mathcal{O}$-estimates of the error terms $P_{p}$ (generalization of Proposition \ref{prop2022}), particularly for the cases $0<p<1$ (for which only rough evaluation formulas have been obtained) by investigating the following functions (the functions (\ref{D}) generalized by $p$-norm $|\xi|_{p}(:=(|\xi_{1}|^{p}+|\xi_{2}|^{p})^{\frac{1}{p}})$).
   \begin{equation}\label{Dp}
    D_{\beta}^{[p]}(s:x):=\frac{1}{\Gamma(\beta+1)}\sum_{|m|_{p}^{p}<s}(s-|m|_{p}^{p})^{\beta}e^{2\pi ix\cdot m},\quad 
    \mathcal{D}_{\beta}^{[p]}(s:x):=\frac{1}{\Gamma(\beta+1)}\int_{|\xi|_{p}^{p}<s}
    (s-|\xi|_{p}^{p})^{\beta}e^{2\pi ix\cdot \xi}d\xi.
  \end{equation}
  \hspace{13pt}Then, we attempt to conduct a study based on the following series expansion (generalization of (\ref{D-J})), which consists of our generalized Bessel functions of order $\omega\geq0$ (\cite{K1}, Definition 2.5).
    \begin{equation}\label{I-Om}
      J_{\omega}^{[p]}(\eta):=
      \begin{cases}
        \frac{1}{\Gamma^{2}(\frac{1}{p})}\left(\frac{2}{p}
    \right)^{2}\int_{0}^{1}\cos(\eta_{1}t^{\frac{1}{p}})\cos(\eta_{2}(1-t)^{\frac{1}
    {p}})t^{\frac{1}{p}-1}(1-t)^{\frac{1}{p}-1}dt\qquad\text{if }\omega=0,\\
        \frac{|\eta|_{p}^{\omega}}{p^{\omega-1}\Gamma(\omega)}\int_{0}^{1}
        J_{0}^{[p]}(\tau\eta)\tau(1-\tau^{p})^{\omega-1}
        d\tau\qquad\text{if }\omega>0.
      \end{cases}
    \end{equation}
  \begin{theorem}[\itshape{\cite{K1}, Theorem 1.3}]\label{thm}
  \itshape{Let $p>0$. If $\beta>-1$ satisfies that $\mathcal{D}_{\beta}^{[p]}(1:x)$ is integrable on $\mathbb{R}^{2}$, then}\vspace{-8pt}
  \begin{equation*}
    D_{\beta}^{[p]}(s:x)-\mathcal{D}_{\beta}^{[p]}(s:x)=s^{\beta+\frac{2}{p}}p^{\beta+1}
    \Gamma^{2}(\frac{1}{p})\sum_{n\in\mathbb{Z}^{2}\setminus\{0\}}\frac{J_{\beta+1}^{[p]}
    (2\pi\sqrt[p]{s}(x-n))}{(2\pi \sqrt[p]{s}|x-n|_{p})^{\beta+1}}
    \qquad\text{for }s>0,\ x\in\mathbb{R}^{2}.\vspace{-5pt}
  \end{equation*}
  Furthermore, under this assumption, the series converges absolutely for $x\in\mathbb{T}^{2}(:=(-\frac{1}{2},\frac{1}{2}]^{2})$.
  \end{theorem}
  For this difficult goal, it is necessary to investigate various properties of $J_{\omega}^{[p]}$, and in this paper we obtain the following asymptotic evaluations (including the well-known evaluation $J_{0}(r)=\mathcal{O}(r^{-\frac{1}{2}})$), which are particularly important properties.
   
  \begin{theorem}[\itshape{Uniformly asymptotic estimates on compact sets on quadrants}]\label{cpt.uni}\ \\
  \itshape{Let $0<p<1$ or $p=2$. For $\eta\ (:=(\eta_{1},\eta_{2}))\in\mathbb{R}^{2}$, if we denote 
  \begin{equation*}
    \eta_{1}=\mathrm{sgn}(\cos\varphi)|\eta|_{p}|\cos\varphi|^{\frac{2}{p}},\quad
    \eta_{2}=\mathrm{sgn}(\sin\varphi)|\eta|_{p}|\sin\varphi|^{\frac{2}{p}},\quad
    (0\leq\varphi<2\pi)
  \end{equation*}
  then the following holds uniformly with respect to $\varphi$
   in any compact sets on $[0,2\pi)\setminus\{0,\frac{\pi}{2},\pi,\frac{3}{2}\pi\}$.}
  \begin{equation*}
    J_{0}^{[p]}(\eta)=\mathcal{O}(|\eta|_{p}^{-\frac{1}{2}})\qquad\text{as }|\eta|_{p}\to\infty.
  \end{equation*}
  \end{theorem}
  
  \begin{theorem}[\itshape{Uniformly asymptotic estimates on $\mathbb{R}^{2}$}]\label{uni}\ \\
    \itshape{For the cases such that $\frac{2}{p}$ are the natural numbers other than $2$, the following holds uniformly with respect to $\varphi\in[0,2\pi)$.}\vspace{-5pt}
    \begin{equation*}\vspace{3pt}
      J_{0}^{[p]}(\eta)=
      \begin{cases}
      \mathcal{O}(|\eta|_{p}^{-\frac{1}{2}}) & (p=2),\\
      \mathcal{O}(|\eta|_{p}^{-\frac{p}{2}}) & (\frac{2}{p}\in\mathbb{N}\setminus{\{1,2\}}),
      \end{cases}
      \qquad\text{as }|\eta|_{p}\to\infty.
    \end{equation*}
  \end{theorem}
  
  \begin{remark}
    Among the generalized Bessel functions $J_{\omega}^{[p]}$ defined by (\ref{I-Om}), particularly in the case $\omega=0$, that is, $J_{0}^{[p]}$ has already been studied by D. St. P. Richards\cite{Richards-1985}\cite{Richards-1986} and W. zu Castell\cite{Castell} from different perspectives than number theory. Theorems \ref{cpt.uni} and \ref{uni} of this paper clarify the (uniformly) asymptotic behaviors for the cases $0<p<1$.
  \end{remark}
  Thus, in particular, the result Theorem \ref{uni} shows that uniformly asymptotic evaluations on $\mathbb{R}^{2}$ for $J_{\omega}^{[p]}$ with $\omega=0$ are obtained for the cases $p$ such that $\frac{2}{p}$ are natural numbers other than $2$. \par
  Therefore, in tackling the lattice point problem of the $p$-circle (including the astroid of Figure 1) for the countable infinite cases $p$, among ones less than 1 (that is, the unsolved cases), the trial by using the functions $J_{\omega}^{[p]}$ and the series representations of the generalized lattice point error functions (Theorem \ref{thm}) is expected to be suitable for the current situation. \par
  Thus, it is clear that \textit{$p$ such that uniformly asymptotic evaluations of $J_{0}^{[p]}$ are obtained} coincide with \textit{$p$ corresponding to some $p$-circle that are not yet solved as lattice point problem}. On the other hand, uniformly asymptotic evaluations of $J_{\omega}^{[p]}$ for positive order $\omega$ are required in the trial for the purpose of solving the problem. \par
  More specifically, it is sufficient to obtain uniformly asymptotic evaluations of $J_{\omega}^{[p]}$ such that $\omega$ are natural numbers (in particular $\omega=1,2$). Note, however, that we cannot obtain uniformly asymptotic evaluations of $J_{\omega}^{[p]}$ in the cases $\omega>0$ by the same method (the proof method in Section 2) as the one that leads to the results via the oscillatory integral representation of $J_{0}^{[p]}$ (Proposition \ref{prop-osc}).\par
  \vspace{5pt}
  Therefore, after proving the two main results of this paper in Section 2, we conclude the paper with some future plan that the one is the problem organization and the importance of the oscillatory integral representations of $J_{\omega}^{[p]}$, and the other is a plan for the derivation of the oscillatory integral representations, in Section 3. 
  
  \section{Proof of Theorems \ref{cpt.uni} -\ref{uni} (Asymptotic behavior)}
  \hspace{13pt}In this section, we derive the uniformly asymptotic evaluations, important properties of the generalized Bessel function $J_{0}^{[p]}$.\par 
  Since $J_{0}^{[p]}$ is a bivariate function, we cannot follow the H. Hankel's complex analytic method \cite{Hankel} which derives asymptotic expansions of the Bessel functions. On the other hand, the asymptotic behavior of the objective can be investigated by considering the oscillatory integral as a real analytical approach. 
  Therefore, we firstly give an oscillatory integral representation of $J_{0}^{[p]}$.
  \begin{proposition}(\itshape{Oscillatory integral representations})\label{prop-osc}\ \\
    \itshape{Let $p>0$. For $\eta\ (:=(\eta_{1},\eta_{2}))\in\mathbb{R}^{2}$, if we define
    \begin{align*}
    f_{p}(\eta,\theta)&:=\eta_{1}\cos^{\frac{2}{p}}\theta+\eta_{2}\sin^{\frac{2}{p}}
    \theta,\qquad g_{p}(\eta,\theta):=\eta_{1}\sin^{\frac{2}{p}}\theta-\eta_{2}
    \cos^{\frac{2}{p}}\theta,\\
    I^{[p]}_{h,\pm}(\eta)&:=\int_{0}^{\frac{\pi}{2}}e^{i(\pm{h_{p}}(\eta,\theta))}\psi^{[p]}
    (\theta)d\theta,\quad\text{with }\psi^{[p]}(\theta):=(\cos\theta\sin\theta)^{\frac{2}{p}-1},
    \quad h:=f,g,
    \end{align*}
    then the following holds.
    \begin{equation}\label{p-osc}
      J_{0}^{[p]}(\eta)=\frac{2}{(p\Gamma(\frac{1}{p}))^{2}}(I^{[p]}_{f,+}(\eta)+I^{[p]}_{f,-}(\eta)+I^{[p]}_{g,+}(\eta)+I^{[p]}_{g,-}(\eta)).
    \end{equation}
    In particular, if $\frac{2}{p}\in\mathbb{N}$ is odd, then it can be shown by the following concise integral.}
    \begin{equation}\label{odd-osc}
      J_{0}^{[p]}(\eta)=\frac{2}{(p\Gamma(\frac{1}{p}))^{2}}\int_{0}^{2\pi}
      e^{i(\eta_{1}\sin^{\frac{2}{p}}\theta+\eta_{2}\cos^{\frac{2}{p}}\theta)}\psi^{[p]}
      (\theta)d\theta.
    \end{equation}
  \end{proposition}
  \vspace{3pt}
  
  \begin{proof}
  From the definition (\ref{I-Om}), the oscillatory integral representation (\ref{p-osc}) can be obtained by deformation as follows.
  \begin{align*}
    J_{0}^{[p]}(\eta)&=\frac{1}{\Gamma^{2}(\frac{1}{p})}\left(\frac{2}{p}\right)^{2}\int_{0}^{1}\cos(\eta_{1}t^{\frac{1}{p}})\cos(\eta_{2}(1-t)^{\frac{1}{p}})t^{\frac{1}{p}-1}(1-t)^{\frac{1}{p}-1}dt\\
    &=\frac{2}{(p\Gamma(\frac{1}{p}))^{2}}\int_{0}^{1}\left(\cos(\eta_{1}t^{\frac{1}{p}}+\eta_{2}(1-t)^{\frac{1}{p}})+\cos(\eta_{1}t^{\frac{1}{p}}-\eta_{2}(1-t)^{\frac{1}{p}})\right)(1-t)^{\frac{1}{p}-1}t^{\frac{1}{p}-1}dt\\
    &=\frac{2}{(p\Gamma(\frac{1}{p}))^{2}}\int_{0}^{\frac{\pi}{2}}\left(\cos(\eta_{1}\cos^{\frac{2}{p}}\theta+\eta_{2}\sin^{\frac{2}{p}}\theta)+\cos(\eta_{1}\sin^{\frac{2}{p}}\theta-\eta_{2}\cos^{\frac{2}{p}}\theta)\right)(\cos^{2}\theta)^{\frac{1}{p}-1}(\sin^{2}\theta)^{\frac{1}{p}-1}2\sin\theta\cos\theta d\theta\\
    &=\frac{2}{(p\Gamma(\frac{1}{p}))^{2}}\int_{0}^{\frac{\pi}{2}}\bigl(e^{i(\eta_{1}
    \cos^{\frac{2}{p}}\theta+\eta_{2}\sin^{\frac{2}{p}}\theta)}+e^{-i(\eta_{1}\cos^{\frac{2}{p}}\theta+\eta_{2}\sin^{\frac{2}{p}}\theta)}+e^{i(\eta_{1}
    \sin^{\frac{2}{p}}\theta-\eta_{2}\cos^{\frac{2}{p}}\theta)}+\\
    &\hspace{230pt}+e^{-i(\eta_{1}\sin^{\frac{2}{p}}\theta-\eta_{2}\cos^{\frac{2}{p}}
    \theta)}\bigr)(\cos\theta\sin\theta)^{\frac{2}{p}-1} d\theta\\
    &=\frac{2}{(p\Gamma(\frac{1}{p}))^{2}}(I^{[p]}_{f,+}(\eta)+I^{[p]}_{f,-}
      (\eta)+I^{[p]}_{g,+}(\eta)+I^{[p]}_{g,-}(\eta)).
  \end{align*}
    \hspace{13pt}Furthermore, from the display (\ref{p-osc}), by applying integral transformations as $t:=\frac{\pi}{2}-\theta$ for $I^{[p]}_{f,+}$, $t:=\frac{3}{2}\pi-\theta$ for $I^{[p]}_{f,-}$, $t:=\pi-\theta$ for $I^{[p]}_{g,+}$, and $t:=2\pi-\theta$ for $I^{[p]}_{g,-}$, it can be expressed as
    \begin{align*}
      J_{0}^{[p]}(\eta)&=\frac{2}{(p\Gamma(\frac{1}{p}))^{2}}\Bigl(\int_{0}^{\frac{\pi}
      {2}}e^{i(\eta_{1}\sin^{\frac{2}{p}}t+\eta_{2}\cos^{\frac{2}{p}}t)}\psi^{[p]}(t)dt+
      \int_{\pi}^{\frac{3}{2}\pi}e^{-i(\eta_{1}(-\sin t)^{\frac{2}{p}}+\eta_{2}
      (-\cos t)^{\frac{2}{p}})}\psi^{[p]}(t)dt\\
      &\hspace{10pt}+\int_{\frac{\pi}{2}}^{\pi}e^{i(\eta_{1}(\sin t)^{\frac{2}{p}}-
      \eta_{2}(-\cos t)^{\frac{2}{p}})}(-\cos t\sin t)^{\frac{2}{p}-1}dt
      +\int_{\frac{3}{2}\pi}^{2\pi}e^{-i(\eta_{1}(-\sin t)^{\frac{2}{p}}-\eta_{2}
      (\cos t)^{\frac{2}{p}})}(-\cos t\sin t)^{\frac{2}{p}-1}dt\Bigr).
    \end{align*}
    In particular, if $\frac{2}{p}$ is odd, then this is the representation (\ref{odd-osc}).
  \end{proof}
  \begin{remark}
    Since $J_{\omega}^{[2]}(\eta)=J_{\omega}(|\eta|)$ and in particular $J_{0}^{[2]}(r,0)=J_{0}(r)$ hold, (\ref{odd-osc}) is a generalization of the following oscillatory integral representation of the Bessel function $J_{0}$ (\cite{Stein-1993}, p338; (13)).
    \begin{equation*}
      J_{0}(r)=\frac{1}{2\pi}\int_{0}^{2\pi}e^{ir\sin\theta}d\theta\qquad\text{for }r>0.
    \end{equation*}
  \end{remark}
  
  \subsection{\normalsize{Proof of Theorem \ref{cpt.uni}}}
  \hspace{13pt}Firstly, we prepare two necessary lemmas to prove Theorem \ref{cpt.uni}.
  \begin{lemma}\label{-1order}
    \itshape{For a real-valued $C^{2}$ function $\phi$ and $C^{1}$ function $\psi$ on the interval} \rm{[}$a,b$\rm{]},\itshape{ if $\phi'(x)\neq0$ on the same interval, then the following holds.}
    \begin{equation*}
      \int_{a}^{b}e^{i\lambda\phi(x)}\psi(x)dx=\mathcal{O}(\lambda^{-1})
      \quad \text{as}\ \lambda\to\infty.
    \end{equation*}
  \end{lemma}
  \begin{proof}
    Since it can be expressed as $e^{i\lambda\phi(x)}=\frac{1}{i\lambda\phi'(x)}(e^{i\lambda\phi(x)})'$ from the assumption, it is clear by integration by parts.
    \begin{align*}
      \int_{a}^{b}e^{i\lambda\phi(x)}\psi(x)dx&=\frac{1}{i\lambda}\int_{a}^{b}
      (e^{i\lambda\phi(x)})'\frac{\psi(x)}{\phi'(x)}dx\\
      &=\frac{1}{i\lambda}\Bigl(\Bigl[e^{i\lambda\phi(x)}\frac{\psi(x)}{\phi'(x)}
      \Bigr]_{a}^{b}-\int_{a}^{b}e^{i\lambda\phi(x)}\Bigl(\frac{\psi(x)}{\phi'(x)}\Bigr)'
      dx\Bigr)=\mathcal{O}(\lambda^{-1})\quad\text{as }\lambda\to\infty.
    \end{align*}
  \end{proof}
  \begin{remark}
    Considering the phase function of Lemma \ref{-1order} as a bivariate $C^{2}$ function $\phi(x,y)$, it is clear from this Lemma's proof that the evaluation formula 
    \begin{equation*}
      \int_{a}^{b}e^{i\lambda\phi(x,y)}\psi(x)dx=\mathcal{O}(\lambda^{-1})\quad
      \text{as }\lambda\to\infty
    \end{equation*}
    uniformly holds with respect to $y$ on sets which satisfy $\frac{\partial}{\partial x}\phi(x,y)\neq0$ and are compact for $y$.
  \end{remark}
  \begin{lemma}[\itshape{\cite{Bleistein-1986}, p220; (6.1.5) or \cite{Bleistein-1984}, p79-p80; 
  (2.7.12), (2.7.17), (2.7.18)}]\label{Bleistein}\ \\
    \itshape{For a real-valued $C^{k}$($k=2$ or $3$) function $\phi$ and $C^{1}$ function $\psi$ on the interval $[a,b]$, if there exists only one $x_{0}\in[a,b]$ satisfying 
    \begin{equation*}
      \phi'(x_{0})=0,\qquad\phi''(x_{0})\neq0,
    \end{equation*} 
    then, as $\lambda\to\infty$, each of the following holds}. \\ 
    (1) If $x_{0}=a$ or $x_{0}=b$, and $k=2$, then
    \begin{equation*}
      \int_{a}^{b}e^{i\lambda\phi(x)}\psi(x)dx=\frac{1}{2}e^{i\lambda{\phi}(x_{0})+
      \frac{\pi}{4}i\ \mathrm{sgn}(\phi''(x_{0}))}\psi(x_{0})\sqrt{\frac{2\pi}
      {\lambda|\phi''(x_{0})|}}+\mathcal{O}(\lambda^{-1}).
    \end{equation*}
    \hspace{15pt}If $x_{0}=a$ or $x_{0}=b$, and $k=3$, as $C_{x_{0}}:=1$ ($x_{0}=a$), $-1$ ($x_{0}=b$), then
    \begin{align*}
      \int_{a}^{b}e^{i\lambda\phi(x)}\psi(x)dx=\frac{1}{2}e^{i\lambda{\phi}(x_{0})+
      \frac{\pi}{4}i\ \mathrm{sgn}(\phi''(x_{0}))}&\Bigl(\psi(x_{0})\sqrt{\frac{2\pi}
      {\lambda|\phi''(x_{0})|}}\\
      &+C_{x_{0}}\frac{2}{\lambda|\phi''(x_{0})|}\Bigl[
      \psi'(x_{0})-\frac{\phi'''(x_{0})}{3\phi''(x_{0})}\Bigr]e^{\frac{\pi}{4}i\ 
      \mathrm{sgn}(\phi''(x_{0}))}\Bigr)+\mathcal{O}(\lambda^{-\frac{3}{2}}).
    \end{align*}
    (2) If $x_{0}\in(a,b)$, then
    \begin{equation*}
      \int_{a}^{b}e^{i\lambda\phi(x)}\psi(x)dx=e^{i\lambda{\phi}(x_{0})+\frac{\pi}{4}i\ 
      \mathrm{sgn}(\phi''(x_{0}))}\psi(x_{0})\sqrt{\frac{2\pi}{\lambda|\phi''(x_{0})|}
      }+\mathcal{O}(\lambda^{-\frac{3}{2}}).
    \end{equation*}
  \end{lemma}
  \vspace{3pt}
  \begin{proof}[\text{Proof of Theorem \ref{cpt.uni}}]
    By applying the variable transformation
    \begin{equation*}
      \eta_{1}=\mathrm{sgn}(\cos\varphi)|\eta|_{p}|\cos\varphi|^{\frac{2}{p}},\quad
      \eta_{2}=\mathrm{sgn}(\sin\varphi)|\eta|_{p}|\sin\varphi|^{\frac{2}{p}},\quad
      (0\leq\varphi<2\pi)
    \end{equation*}
    to Proposition \ref{prop-osc} and replacing the phase functions with 
    \begin{align*}
      f_{p,\varphi}(\theta)&:=
      \begin{cases}
        (\cos\varphi\cos\theta)^{\frac{2}{p}}+(\sin\varphi\sin\theta)^{\frac{2}{p}} & 
        \text{if }0\leq\varphi<\frac{\pi}{2},\\
        -(-\cos\varphi\cos\theta)^{\frac{2}{p}}+(\sin\varphi\sin\theta)^{\frac{2}{p}} & 
        \text{if }\frac{\pi}{2}\leq\varphi<\pi,\\
        -(-\cos\varphi\cos\theta)^{\frac{2}{p}}-(-\sin\varphi\sin\theta)^{\frac{2}{p}} & 
        \text{if }\pi\leq\varphi<\frac{3}{2}\pi,\\
        (\cos\varphi\cos\theta)^{\frac{2}{p}}-(-\sin\varphi\sin\theta)^{\frac{2}{p}} & 
        \text{if }\frac{3}{2}\pi\leq\varphi<2\pi,
      \end{cases}\\
      g_{p,\varphi}(\theta)&:=
      \begin{cases}
        (\cos\varphi\sin\theta)^{\frac{2}{p}}-(\sin\varphi\cos\theta)^{\frac{2}{p}} & 
        \text{if }0\leq\varphi<\frac{\pi}{2},\\
        -(-\cos\varphi\sin\theta)^{\frac{2}{p}}-(\sin\varphi\cos\theta)^{\frac{2}{p}} & 
        \text{if }\frac{\pi}{2}\leq\varphi<\pi,\\
        -(-\cos\varphi\sin\theta)^{\frac{2}{p}}+(-\sin\varphi\cos\theta)^{\frac{2}{p}} & 
        \text{if }\pi\leq\varphi<\frac{3}{2}\pi,\\
        (\cos\varphi\sin\theta)^{\frac{2}{p}}+(-\sin\varphi\cos\theta)^{\frac{2}{p}} & 
        \text{if }\frac{3}{2}\pi\leq\varphi<2\pi,
      \end{cases}
    \end{align*}
    then, we can again express it as follows.
    \begin{align*}
    J_{0}^{[p]}(\eta)=\frac{2}{(p\Gamma(\frac{1}{p}))^{2}}(I^{[p]}_{f,+}(\eta)+&I^{[p]}_{f,-}(\eta)+I^{[p]}_{g,+}(\eta)+I^{[p]}_{g,-}(\eta)),\\
      \text{with }&I^{[p]}_{h,\pm}(\eta)=\int_{0}^{\frac{\pi}{2}}e^{i|\eta|_{p}(\pm{h_{p,\varphi}}(\theta))}
    \psi^{[p]}(\theta)d\theta\quad\text{for }h=f,g.
    \end{align*}
    \hspace{13pt}In the following, we only need to consider a compact set on $0<\varphi<\frac{\pi}{2}$ by symmetry (similar arguments hold for the other quadrants by paying attention to the sign). That is, we fix $\eta$ on such a compact set taken arbitrarily and investigate stationary points of the phase functions
    \begin{equation*}
      f_{p,\varphi}(\theta)=(\cos\varphi\cos\theta)^{\frac{2}{p}}+(\sin\varphi\sin
      \theta)^{\frac{2}{p}},\quad 
      g_{p,\varphi}(\theta)=(\cos\varphi\sin\theta)^{\frac{2}{p}}-(\sin\varphi\cos
      \theta)^{\frac{2}{p}} \quad\text{for }0\leq\theta\leq\frac{\pi}{2}.
    \end{equation*}
    Note that a stationary point is a point where the value of the first-order derivative of the function is $0$.\\\vspace{3pt}\par
    Firstly, we focus on the case $p=2$, where the argument is concise. Then, by
    \begin{equation*}
      f_{2,\varphi}(\theta)=\cos(\theta-\varphi),\quad
      g_{2,\varphi}(\theta)=\sin(\theta-\varphi),
    \end{equation*}
    \begin{equation*}
      f_{2,\varphi}'(\theta)=-\sin(\theta-\varphi),\quad
      g_{2,\varphi}'(\theta)=\cos(\theta-\varphi),
    \end{equation*}
    so the stationary point of $f_{2,\varphi}$ is $\theta=\varphi(\in(0,\frac{\pi}{2}))$, and there is no stationary point of $g_{2,\varphi}$. \par
    From these and $f_{2,\varphi}''(\varphi)=-1$, we obtain the following by appropriately partitioning the integral interval and applying Lemma \ref{-1order} and Lemma \ref{Bleistein}. (The main term of this result is indeed consistent with that of the well-known asymptotic behavior in the conventional Bessel functions. See, for example, (6.1.17), (6.1.18) of \cite{Bleistein-1986} or p199(1) of \cite{Watson}.)
    \begin{align*}
      J_{0}^{[2]}(\eta)&=\frac{1}{2\pi}(e^{(|\eta|-\frac{\pi}{4})i}+e^{(-|\eta|+
      \frac{\pi}{4})i})\sqrt{\frac{2\pi}{|\eta|}}+\mathcal{O}(|\eta|^{-1})\\
      &=\sqrt{\frac{2}{\pi|\eta|}}\cos(|\eta|-\frac{\pi}{4})+\mathcal{O}(|\eta|^{-
      1})\qquad\text{as }|\eta|\to\infty.
    \end{align*}
    \hspace{13pt}Furthermore, since $\mathrm{arg}(\eta)(:=\varphi)$ is arbitrary on the compact set fixed, 
    \begin{equation*}
      J_{0}^{[2]}(\eta)=\mathcal{O}(|\eta|^{-\frac{1}{2}})\qquad\text{as }|\eta|\to\infty
    \end{equation*}
    holds uniformly on this set, and from the arbitrariness of the set, we obtained the desired evaluation formula uniformly on compacts in the case $p=2$. \\
    \vspace{5pt}\par
    Therefore, we will follow above flow to consider the cases $0<p<1$ below.\par
    Since it can be expressed as 
    \begin{equation*}
      f_{p,\varphi}'(\theta)=-\frac{2}{p}\sin\theta\cos\theta(\cos^{\frac{2}{p}}\varphi
      \cos^{\frac{2}{p}-2}\theta-\sin^{\frac{2}{p}}\varphi\sin^{\frac{2}{p}-2}\theta)
      =:-\frac{2}{p}\sin\theta\cos\theta\ u_{\varphi}^{[p]}(\theta),
    \end{equation*}
    it is clear that $\theta=0,\frac{\pi}{2}$ are stationary points of $f_{p,\varphi}$, and we can also find the remaining stationary point as follows. \vspace{-5pt}
    \begin{align*}
      u_{\varphi}^{[p]}(\theta)=0\ &\iff\ \cos^{\frac{2}{p}}\varphi\cos^{\frac{2}{p}-2}\theta
      =\sin^{\frac{2}{p}}\varphi\sin^{\frac{2}{p}-2}\theta\\
      &\iff\ \cos^{\frac{1}{1-p}}\varphi\cos\theta=\sin^{\frac{1}{1-p}}\varphi\sin\theta
      \\
      &\iff\ a(\varphi)(\cos\varphi_{0}\cos\theta-\sin\varphi_{0}\sin\theta)=0\\
      &\iff\ \cos(\theta+\varphi_{0})=0 \qquad \iff\ \theta=\frac{\pi}{2}-\varphi_{0}.
    \end{align*}
    Note, however, that the symbol is defined as 
    \begin{equation*}
      a(\varphi):=(\cos^{\frac{2}{1-p}}\varphi+\sin^{\frac{2}{1-p}}\varphi)^{\frac{1}{2}},
    \end{equation*}
    and by the composition of trigonometric functions, $\varphi_{0}\in(0,\frac{\pi}{2})$ is taken such that it is satisfied 
    \begin{equation}\label{varphi_0}
      \cos\varphi_{0}=\frac{\cos^{\frac{1}{1-p}}\varphi}{a(\varphi)},\quad
      \sin\varphi_{0}=\frac{\sin^{\frac{1}{1-p}}\varphi}{a(\varphi)}.
    \end{equation}
    \par
    Next, investigate $f''_{p,\varphi}$. Since they can be expressed as
    \begin{align*}
      f_{p,\varphi}''(\theta)&=-\frac{2}{p}(\cos2\theta\ u_{\varphi}^{[p]}(\theta)+\sin\theta
      \cos\theta\ u_{\varphi}^{[p]}\hspace{1pt}'(\theta)),\\
      u_{\varphi}^{[p]}\hspace{1pt}'(\theta)&=-\Bigl(\frac{2}{p}-2\Bigr)(\cos^{\frac{2}{p}}\varphi
      \sin\theta\cos^{\frac{2}{p}-3}\theta+\sin^{\frac{2}{p}}\varphi\cos\theta
      \sin^{\frac{2}{p}-3}\theta),
    \end{align*}
    in addition to 
    $f_{p,\varphi}''(0)=-\frac{2}{p}\cos^{\frac{2}{p}}\varphi<0,
    \ f_{p,\varphi}''(\frac{\pi}{2})=-\frac{2}{p}\sin^{\frac{2}{p}}\varphi<0$,
    the following can be confirmed.
    \begin{align*}
      f_{p,\varphi}''(\frac{\pi}{2}-\varphi_{0})&=\frac{2}{p}\Bigl(\frac{2}{p}-2\Bigr)
      \sin(\frac{\pi}{2}-\varphi_{0})\cos(\frac{\pi}{2}-\varphi_{0})\Bigl(\cos^{\frac{2}{p}}\varphi\sin(\frac{\pi}{2}-\varphi_{0})\cos^{\frac{2}{p}-3}
      (\frac{\pi}{2}-\varphi_{0})\\
      &\hspace{30pt}+\sin^{\frac{2}{p}}\varphi\cos(\frac{\pi}{2}-\varphi_{0})
      \sin^{\frac{2}{p}-3}(\frac{\pi}{2}-\varphi_{0})\Bigr)\\
      &=\frac{2}{p}\Bigl(\frac{2}{p}-2\Bigr)\cos\varphi_{0}\sin\varphi_{0}(\cos^{\frac{2}
      {p}}\varphi\cos\varphi_{0}\sin^{\frac{2}{p}-3}\varphi_{0}+\sin^{\frac{2}{p}}\varphi
      \sin\varphi_{0}\cos^{\frac{2}{p}-3}\varphi_{0})>0.
    \end{align*}
    \par
    In addition to the above and $\psi^{[p]}(0)=\psi^{[p]}(\frac{\pi}{2})=0$, we obtain the following by applying Lemma \ref{-1order} and Lemma \ref{Bleistein} after dividing the integral interval appropriately as in the case $p=2$.
    \begin{equation*}
      I^{[p]}_{f,+}(\eta)+I^{[p]}_{f,-}(\eta)=
        \frac{2\sqrt{2\pi}\cos(|\eta|_{p}f_{p,\varphi}(\frac{\pi}{2}-\varphi_{0})+
        \frac{\pi}{4})
        \psi^{[p]}(\frac{\pi}{2}-\varphi_{0})}{\sqrt{f_{p,\varphi}''(\frac{\pi}{2}-
        \varphi_{0})}}|\eta|_{p}^{-\frac{1}{2}}+\mathcal{O}(|\eta|_{p}^{-1}) \quad
        \text{as }|\eta|_{p}\to\infty.
    \end{equation*}
    \hspace{13pt}Furthermore, $\tilde{f}_{p}(\varphi):=f_{p,\varphi}''(\frac{\pi}{2}-\varphi_{0})$ is continuous from (\ref{varphi_0}) and positive on the assumed compact set, so $\tilde{f}_{p}(\varphi)$ has the minimum value greater than $0$ on the same set. Therefore, we obtain a uniform evaluation formula
    \begin{equation}\label{f_cpt.uni}
      I^{[p]}_{f,+}(\eta)+I^{[p]}_{f,-}(\eta)=
        \mathcal{O}(|\eta|_{p}^{-\frac{1}{2}}) \quad\text{as }|\eta|_{p}\to\infty, \text{ if } 0<p<1
    \end{equation}
    on this compact set. \par
    From now on, it is sufficient to consider the evaluation of $I^{[p]}_{g,\pm}$ for the cases $0<p<1$ in the same way. By 
    \begin{align*}
      g_{p,\varphi}'(\theta)=\frac{2}{p}\sin\theta\cos\theta(&\cos^{\frac{2}{p}}\varphi
      \sin^{\frac{2}{p}-2}\theta+\sin^{\frac{2}{p}}\varphi\cos^{\frac{2}{p}-2}\theta)
      =:\frac{2}{p}\sin\theta\cos\theta\ v_{\varphi}^{[p]}(\theta),\\
      v_{\varphi}^{[p]}(\theta)=0\ &\iff\ \cos^{\frac{2}{p}}\varphi\sin^{\frac{2}{p}-2}\theta
      =0,\ \sin^{\frac{2}{p}}\varphi\cos^{\frac{2}{p}-2}\theta=0,\\
      &\iff\ \sin\theta=\cos\theta=0,
    \end{align*}
    we find that there is no such $\theta$ and the only stationary points of $g_{p,\varphi}$ are $\theta=0,\frac{\pi}{2}$.\par
    Thus, from $\psi^{[p]}(0)=\psi^{[p]}(\frac{\pi}{2})=0$, we can also apply Lemma \ref{-1order} and Lemma \ref{Bleistein}, and again combine them with the evaluation formulas (\ref{f_cpt.uni}) to obtain the overall evaluation formulas uniformly on compacts, that is, to complete the proof.
    \begin{equation*}
      J_{0}^{[p]}(\eta)=\mathcal{O}(|\eta|_{p}^{-\frac{1}{2}})
      \qquad\text{as }|\eta|_{p}\to\infty.
    \end{equation*}
  \end{proof}
  \begin{remark}
    For the cases $2<p$ or $1<p<2$, the asymptotic expansion Lemma \ref{Bleistein} cannot be applied to obtain evaluation formulas uniformly on compacts because the symbol function $\psi^{[p]}$ and its first-order derivative diverge at the end points which are the stationary points, and the $C^{1}$ class condition is not satisfied. On the other hand, for the case $p=1$, it is found that the desired uniform evaluation formula cannot be obtained.
  \end{remark}
  \subsection{\normalsize{Proof of Theorem \ref{uni}}}
  \hspace{13pt}Furthermore, for the purpose of further higher-order differentiation, we will consider the right-hand neighborhood of axis $\varphi=\frac{\pi}{2}$, where $\frac{2}{p}$ are natural numbers (other than 2) so that the phase functions have appropriate smoothness (by symmetry, the other neighborhood and the other three axes can be discussed similarly). For sufficiently small $0\leq\delta<<1$, defined as the functions 
  \begin{equation*}
    F_{p,\delta}(\theta):=\delta\cos^{\frac{2}{p}}\theta+\sin^{\frac{2}{p}}\theta,\quad
    G_{p,\delta}(\theta):=\delta\sin^{\frac{2}{p}}\theta-\cos^{\frac{2}{p}}\theta\qquad
    \text{for }0\leq\theta\leq\frac{\pi}{2},
  \end{equation*}
  each symbol of Proposition \ref{prop-osc} for a neighborhood point $\eta:=(\delta\lambda,\ \lambda)\ (\lambda>0)$ can be expressed as 
  \begin{align*} 
    f_{p}(\eta,\theta)&=\lambda F_{p,\delta}(\theta),\quad 
    g_{p}(\eta,\theta)=\lambda G_{p,\delta}(\theta),\\
    I^{[p]}_{h,\pm}(\eta)=\int_{0}^{\frac{\pi}{2}}e^{i\lambda(\pm{H_{p,\delta}}(\theta))}
    &\psi^{[p]}(\theta)d\theta,\qquad H=F\text{ if }h=f,\ H=G\text{ if }h=g,
  \end{align*}
  in particular $F_{p,\delta}$ and $G_{p,\delta}$ are phase functions for the oscillatory integral $I^{[p]}_{h,\pm}$. \par
  Then, they can be written as 
  \begin{align}
    F'_{p,\delta}(\theta)&=-\frac{2}{p}\sin\theta\cos\theta(\delta\cos^{\frac{2}{p}-2}
    \theta-\sin^{\frac{2}{p}-2}\theta)=-\frac{2}{p}\sin\theta\cos\theta\ u_{1,\delta}
    (\theta),\label{F'}\\
    G'_{p,\delta}(\theta)&=\frac{2}{p}\sin\theta\cos\theta(\delta\sin^{\frac{2}{p}-2}
    \theta+\cos^{\frac{2}{p}-2}\theta)=:\frac{2}{p}\sin\theta\cos\theta\ v_{\delta}
    (\theta)\notag.
  \end{align}
  Note that we defined $u_{k,\delta}(\theta):=\delta\cos^{\frac{2}{p}-2k}
  \theta-\sin^{\frac{2}{p}-2k}\theta$ in general.\vspace{10pt}\par

  $(\ \mathrm{i}\ )$ For the case $p=2$,  
  \begin{equation*}
    F'_{2,\delta}(\theta)=-\delta\sin\theta+\cos\theta,\quad
    G'_{2,\delta}(\theta)=\delta\cos\theta+\sin\theta,
  \end{equation*}
  thus $F_{2,0}$ has only one stationary point, and $G_{2,\delta}$ has a stationary point $\theta=0$ only with $\delta=0$. \par
  In fact, 
  \begin{align*}
    F_{2,0}'(\theta)(=\cos\theta)=0\ &\iff\ \theta=\frac{\pi}{2},\\
    G_{2,\delta}'(\theta)=0\ &\iff\ \sin\theta=-\delta\cos\theta\ \iff\ 
    \delta=0,\ \theta=0
  \end{align*}
  hold, while on $\delta>0$ (that is, off-axis), from the composition of trigonometric functions, by taking $\theta_{\delta}\in(0,\frac{\pi}{2})$ such that 
  \begin{equation}\label{p=2_theta_delta}
    \cos\theta_{\delta}=\frac{1}{a_{2}(\delta)},\quad
    \sin\theta_{\delta}=\frac{\delta}{a_{2}(\delta)}\qquad
    \text{with }a_{2}(\delta):=\sqrt{\delta^{2}+1}
  \end{equation}
  are satisfied, we can confirm the following.
  \begin{equation*}
    F_{2,\delta}'(\theta)=0\ \iff\ a_{2}(\delta)(\cos\theta\cos\theta_{\delta}-
    \sin\theta\sin\theta_{\delta})=0\ \iff\ \cos(\theta+\theta_{\delta})=0 \  
    \iff\ \theta=\frac{\pi}{2}-\theta_{\delta}.
  \end{equation*}
  \par
  $(\ \mathrm{ii}\ )$ For the cases $0<p<1$, 
  \begin{align*}
    u_{1,0}(\theta)(=-\sin^{\frac{2}{p}-2}\theta)=0\ &\iff\ \theta=0,\\ 
    v_{0}(\theta)(=\cos^{\frac{2}{p}-2}\theta)=0\ &\iff\ \theta=\frac{\pi}{2},
  \end{align*}
  thus the stationary points for $F_{p,0}$ and $G_{p,0}$ are only $0$ and $\frac{\pi}{2}$, and while on $\delta>0$, as above, from the composition of trigonometric functions, by taking $\theta_{\delta}\in(0,\frac{\pi}{2})$ such that 
  \begin{equation}\label{0<p<1_theta_delta}
    \cos\theta_{\delta}=\frac{1}{a_{p}(\delta)},\quad
    \sin\theta_{\delta}=\frac{\delta^{\frac{-p}{2(1-p)}}}{a_{p}(\delta)}\qquad
    \text{with }a_{p}(\delta):=\sqrt{\delta^{\frac{-p}{1-p}}+1}
  \end{equation}
  are satisfied, we can confirm the following.
  \begin{align*}
    u_{1,\delta}(\theta)=0\ &\iff\ \delta^{\frac{p}{2(1-p)}}\cos\theta=\sin\theta\ \iff\ 
    -\delta^{\frac{-p}{2(1-p)}}\sin\theta+\cos\theta=0\\
    &\iff\ a_{p}(\delta)(\cos\theta\cos\theta_{\delta}-
    \sin\theta\sin\theta_{\delta})=0\ \iff\ \cos(\theta+\theta_{\delta})=0 \  
    \iff\ \theta=\frac{\pi}{2}-\theta_{\delta}.
  \end{align*}
  \hspace{13pt}Hence, the stationary points of $F_{p,\delta}$ are the three points of $0,\ \frac{\pi}{2}-\theta_{\delta}$ and $\frac{\pi}{2}$, while the statonary points of $G_{p,\delta}$ are only the two points of $0$ and $\frac{\pi}{2}$, from $v_{\delta}(\theta)>0$.\vspace{-20pt}\\
  \begin{table}[h]
    \caption{Stationary points for each the phase functions}
      \centering
      \begin{tabular}{|l|c|c|c|c|} \hline
         & $p=2\ (\delta=0)$ & $p=2\ (\delta>0)$ & $0<p<1,\ \frac{2}{p}
         \in\mathbb{N}\ (\delta=0)$ & $0<p<1,\ \frac{2}{p}\in\mathbb{N}\ (\delta>0)$ 
         \\ \hline
        $F_{p,\delta}$ & $\frac{\pi}{2}$ & $\frac{\pi}{2}-\theta_{\delta}$ & $0,\ \frac{\pi}{2}$ & $0,\ \frac{\pi}{2}-\theta_{\delta},\ 
        \frac{\pi}{2}$ \\
        $G_{p,\delta}$ & $0$ & None & $0,\ \frac{\pi}{2}$ & $0,\ 
        \frac{\pi}{2}$ \\ \hline
      \end{tabular}
  \end{table}
  \par
  Now, while we have obtained the stationary points in each case, we need to introduce the following lemma in addition to these in order to prove Theorem \ref{uni}.
  \begin{lemma}[\itshape{Van der Corput's Lemma: \cite{Duoandikoetxea}, Lemma 8.27. or 
  \cite{Stein-1993}, p334; Corollary}]\label{-1/k_order}\ \\
  \itshape{For a real-valued $C^{k}$($k\geq2$) function $\phi$ and $C^{1}$ function $\psi$ on the interval $[a,b]$, if $|\phi^{(k)}(x)|\geq1$ on the same interval, then the following holds.}
    \begin{equation*}
      \int_{a}^{b}e^{i\lambda\phi(x)}\psi(x)dx=\mathcal{O}(\lambda^{-\frac{1}{k}})\qquad
      \text{as }\lambda\to\infty.
    \end{equation*}
  \end{lemma}
  We note that, as shown in Table 1, for each phase function, there is a stationary point $\frac{\pi}{2}-\theta_{\delta}$ that depends on $\delta$ and are stationary points $0$ and $ \frac{\pi}{2}$ that are fixed independent of $\delta$, and we firstly focus on the latter. \par\vspace{5pt}
  We will see later the case $p=2$. In the case $0<p<1$ satisfying $\frac{2}{p}\in\mathbb{N}$, since the derivatives of 
  \begin{equation*}
    F_{p,0}(\theta)=\sin^{\frac{2}{p}}\theta,\quad G_{p,0}(\theta)=-\cos^{\frac{2}{p}}
    \theta\qquad\text{for }0\leq\theta\leq\frac{\pi}{2}
  \end{equation*}
  up to the $\frac{2}{p}$-th order can be expressed as 
  \begin{align*}
    F_{p,0}^{(n)}(\theta)&=\sum_{j=0}^{n}b_{n,j}\Bigl(\frac{2}{p},\cos\theta\Bigr)
    \sin^{\frac{2}{p}-j}\theta,\qquad G_{p,0}^{(n)}(\theta)=-\sum_{j=0}^{n}c_{n,j}\Bigl(
    \frac{2}{p},\sin\theta\Bigr)\cos^{\frac{2}{p}-j}\theta,\\
    b_{n,n}\Bigl(\frac{2}{p},\cos\theta\Bigr)&=
    \begin{cases}
      (\frac{2}{p})_{n}\cos\theta & \text{if }n\text{ : odd},\\
      (\frac{2}{p})_{n} & \text{if }n\text{ : even},
    \end{cases}
    \qquad
    c_{n,n}\Bigl(\frac{2}{p},\sin\theta\Bigr)=
    \begin{cases}
      (-1)^{n}(\frac{2}{p})_{n}\sin\theta & \text{if }n\text{ : odd},\\
      (-1)^{n}(\frac{2}{p})_{n} & \text{if }n\text{ : even},
    \end{cases}
  \end{align*}
  the following holds (note $(m)_{n}:=\prod_{j=0}^{n-1}(m-j)$).
  \begin{equation*}
    F_{p,0}^{(\frac{2}{p})}(0)=\Bigl(\frac{2}{p}\Bigr)!\neq0,\qquad
    G_{p,0}^{(\frac{2}{p})}(\frac{\pi}{2})=(-1)^{\frac{2}{p}+1}\Bigl(\frac{2}{p}\Bigr)!
    \neq0,
  \end{equation*}
  \begin{equation*}
    F_{p,0}''(\frac{\pi}{2})=-\frac{2}{p}\neq0,\qquad
    G_{p,0}''(0)=\frac{2}{p}\neq0.
  \end{equation*}
  Further from these, from the continuity of $F_{p,\delta}^{(n)}$ and $G_{p,\delta}^{(n)}$ with respect to $\theta$ and $\delta$, we can take $0<a<b<\frac{\pi}{2}$ and $0<\delta'<<1$ such that
  \begin{align}
    F_{p,\delta}^{(\frac{2}{p})}(\theta)&\neq0,\quad G_{p,\delta}''(\theta)\neq0,\qquad 
    \text{for }0\leq\theta\leq a,\ 0\leq\delta\leq\delta'\label{theta_a}\\
    G_{p,\delta}^{(\frac{2}{p})}(\theta)&\neq0,\quad F_{p,\delta}''(\theta)\neq0\ \ 
    \qquad\text{for }b\leq\theta\leq\frac{\pi}{2},\ 0\leq\delta\leq\delta'\label{theta_b},
  \end{align}
  are satisfied (note that $a,\ b$ and $\delta'$ are independent of each other). 
  Then, if we note the relationship between $|\eta|_{p}\backsimeq\lambda$ (that is, constant multiple) and apply Lemma \ref{-1/k_order} (on the integral intervals $[0,a],\ [b,\frac{\pi}{2}]$) and Lemma \ref{-1order} (on the integral interval $[a,b]$) to $0\leq\delta\leq\delta'$, we can obtain uniform evaluation formulas with respect to $0\leq\delta\leq\delta'$. 
  \begin{equation}
    I^{[p]}_{g,\pm}(\eta)=\mathcal{O}(|\eta|_{p}^{-\frac{p}{2}})\quad\text{as }|\eta|_{p}\to\infty,\text{ if }0<p<1\text{ such that }\frac{2}{p}\in\mathbb{N}.\label{0<p<1_g}
  \end{equation}
  \vspace{3pt}\par
  On the other hand, we need to also turn our attention to the cases where the remaining stationary point $\frac{\pi}{2}-\theta_{\delta}$, that is, the point dependent on $\delta$ appears (see Table 1). Firstly, we consider the case $p=2$ postponed above.\par
  For $F_{2,\delta}(\theta)=\delta\cos\theta+\sin\theta$ with $\delta>0$, from the way $\theta_{\delta}$ is defined (\ref{p=2_theta_delta}), the following holds.
  \begin{align*}
    F_{2,\delta}''(\theta)&=
    -(\delta\cos\theta+\sin\theta),\\
    F_{2,\delta}''(\frac{\pi}{2}-\theta_{\delta})&=-\delta\sin\theta_{\delta}-\cos
    \theta_{\delta}\\
    &=\frac{-1}{a_{2}(\delta)}(\delta^{2}+1)\neq0\quad(=\mathcal{O},\ \Omega(1)\quad\text{as }\delta\to0,
    \ \text{that is, }\theta_{\delta}\to0).
  \end{align*}\par
  In addition, since, for $G_{2,\delta}(\theta)=\delta\sin\theta-\cos\theta$,
  \begin{equation*}
    G_{2,\delta}''(\theta)=-(\delta\sin\theta-\cos\theta),\qquad G_{2,\delta}''(0)
    =1\neq0\quad(=\mathcal{O},\ \Omega(1)\quad\text{as }\delta\to0),
  \end{equation*}
  hold, we can take sufficiently small $a,\ \delta'>0$ satisfying the following (note $\frac{\pi}{2}-\theta_{\delta'}\leq\frac{\pi}{2}-\theta_{\delta}\leq\frac{\pi}{2}$).
  \begin{align*}
    F_{2,\delta}''(\theta)&\neq0\qquad\text{for }0\leq\delta\leq\delta',\ \frac{\pi}{2}-
    \theta_{\delta'}-a\leq\theta\leq\frac{\pi}{2},\\
    G_{2,\delta}''(\theta)&\neq0\qquad\text{for }0\leq\delta\leq\delta',\ 0\leq\theta\leq a.
  \end{align*}\par
  Therefore, by applying Lemma \ref{-1/k_order} (on the integral intervals $[\frac{\pi}{2}-\theta_{\delta'}-a,\frac{\pi}{2}]$,\ $[0,a]$) and Lemma \ref{-1order} (on the integral intervals $[0,\frac{\pi}{2}-
  \theta_{\delta'}-a]$,\ $[a,\frac{\pi}{2}]$) again, we obtain the following uniform evaluation formulas with respect to $0\leq\delta\leq\delta'$.
  \begin{equation}\label{p=2_h}
    I^{[2]}_{h,\pm}(\eta)=\mathcal{O}(|\eta|^{-\frac{1}{2}})\quad\text{as }|\eta|\to\infty,\text{ if }p=2,\ h=f\text{ or }g.
  \end{equation}
  \vspace{5pt}\par
  Finally, for the cases $0<p<1$ such that $\frac{2}{p}$ is a natural number greater than or equal to $3$, it is sufficient to show that the $\frac{2}{p}$-th derivative of $F_{p,\delta}$ is nonzero as $\delta$ → 0 on the $\delta$-dependent stationary point. In addition to this, we will also clarify the specific speed at which the derivatives of the orders less than $\frac{2}{p}$ approach zero as the following proposition.
  \begin{proposition}\label{Omega_1}
    \itshape{Let $\frac{2}{p}\in\mathbb{N}\setminus{\{1,2\}}$ and $\delta>0$. For $F_{p,\delta}$ and $
    \theta_{\delta}$ defined above, the following holds.} 
    \begin{equation*}
      F_{p,\delta}^{(n)}(\frac{\pi}{2}-\theta_{\delta})=
        \begin{cases}
        0 & \text{if }n=1,\\
        \mathcal{O}\Bigl(\delta^{\frac{\frac{2}{p}-n}{\frac{2}{p}(1-p)}}\Bigr) & 
        \text{if }1<n<\frac{2}{p},\\
        \mathcal{O},\ \Omega(1) & \text{if }n=\frac{2}{p},
        \end{cases}
      \qquad\text{as }\delta\to0.
    \end{equation*}
  \end{proposition}
  \hspace{13pt}\vspace{2pt}\par
  For the present, we admit that this proposition holds, and then we complete the proof of Theorem \ref{uni} by using the aforementioned results, followed by the proof of the proposition.
  \vspace{3pt}
  \begin{proof}[\text{Proof of Theorem \ref{uni}}]
    \hspace{13pt}In the cases $p$ such that $\frac{2}{p}\in\mathbb{N}\setminus{\{1,2\}}$, $\frac{\pi}{2}-\theta_{\delta}\to0$ holds as \ $\delta\to0$ from the way $\theta_{\delta}$ is defined in (\ref{0<p<1_theta_delta}). Then, we re-select $\delta'$ such that $\frac{\pi}{2}-\theta_{\delta}<a\ (\delta\in[0,\delta'])$ for $a$ in (\ref{theta_a}). Furthermore, by applying Lemma \ref{-1/k_order} (on the integral intervals $[0,a],\ [b,\frac{\pi}{2}]$) and Lemma \ref{-1order} (on the integral interval $[a,b]$) to the results ($\ref{theta_a}$), (\ref{theta_b}) and Proposition \ref{Omega_1}, we obtain uniform evaluation formulas
    \begin{equation*}
      I^{[p]}_{f,\pm}(\eta)=\mathcal{O}(|\eta|_{p}^{-\frac{p}{2}})\qquad\text{as }|\eta|_{p}\to\infty,\text{ if }0<p<1,\ \frac{2}{p}\in\mathbb{N}
    \end{equation*}
  with respect to $\delta\in[0,\delta']$. If we put this together with the already obtained evaluation formulas $(\ref{0<p<1_g})$ and $(\ref{p=2_h})$, then we can take $0<\delta'<<1$ for each $p$ satisfying $\frac{2}{p}\in\mathbb{N}$, and obtain uniform evaluation formulas
  \begin{equation*}
      I^{[p]}_{h,\pm}(\eta)=
      \begin{cases}
        \mathcal{O}(|\eta|_{p}^{-\frac{1}{2}}) & \text{if }p=2,\\
        \mathcal{O}(|\eta|_{p}^{-\frac{p}{2}}) & \text{if }0<p<1\text{ such that }\frac{2}{p}\in\mathbb{N},
      \end{cases}
      \qquad\text{as }|\eta|_{p}\to\infty,\text{ for }h=f\text{ or }g
    \end{equation*}
  on the right neighborhood $\{\eta=(\delta\lambda, \lambda)\in\mathbb{R}^{2}|\lambda>0,\ 0\leq\delta\leq\delta'\}$ with respect to the positive $y$-axis $(\mathrm{arg}(\eta)=\frac{\pi}{2})$. That is, from the display of $J_{0}^{[p]}$ (\ref{p-osc}), we obtain uniform evaluation formulas 
    \begin{equation}\label{near_uni}
      J_{0}^{[p]}(\eta)=
      \begin{cases}
        \mathcal{O}(|\eta|_{p}^{-\frac{1}{2}}) & \text{if }p=2,\\
        \mathcal{O}(|\eta|_{p}^{-\frac{p}{2}}) & \text{if }0<p<1\text{ such that }\frac{2}{p}\in\mathbb{N},
      \end{cases}
      \qquad\text{as }|\eta|_{p}\to\infty
    \end{equation}
   on the right-hand neighborhood of the axis $\mathrm{arg}(\eta)=\frac{\pi}{2}$. \par
   On the other hand, recalling that we have already obtained evaluation formulas uniformly on compacts on $\mathrm{arg}(\eta)\in(0,\frac{\pi}{2})$ as Theorem \ref{cpt.uni}, and furthermore, recalling that symmetry led us to the discussion on the first quadrant and the right neighborhood in the positive part of $y$-axis, we can conclude that the evaluation formulas (\ref{near_uni}) hold uniformly on $\mathbb{R}^{2}$ by similar discussion.
  \end{proof}
  \begin{remark}
    If we restrict $p$ to the cases where $\frac{2}{p}\in\mathbb{N}$ is odd, then the uniformly asymptotic evaluations of $J_{0}^{[p]}$ on $\mathbb{R}^{2}$ are obtained by a more concise argument (based on the oscillatory integral representation (\ref{odd-osc}) and the smoothness of the corresponding phase function and its higher-order derivatives) and Theorem \ref{cpt.uni}.
  \end{remark}
  
  \subsection{\normalsize{Proof of Proposition \ref{Omega_1}}}
  \hspace{13pt}Under the symbols and assumptions of Proposition \ref{Omega_1}, we make some preparations. Firstly, from the way the point $\theta_{\delta}$ is defined (\ref{0<p<1_theta_delta}), if we note
  \begin{equation*}
    \cos\theta_{\delta}=(\delta^{-\frac{p}{1-p}}+1)^{-\frac{1}{2}}(=\mathcal{O}(
    \delta^{\frac{p}{2(1-p)}})\quad\text{as }\delta\to0),\quad\sin\theta_{\delta}
    =(\delta^{\frac{p}{1-p}}+1)^{-\frac{1}{2}}(=\mathcal{O}(1)\quad\text{as }\delta\to0),
  \end{equation*}
  then it is clear that the following holds.
  \begin{equation}\label{cos2}
    \cos2(\frac{\pi}{2}-\theta_{\delta})=\sin^{2}\theta_{\delta}-\cos^{2}\theta_{\delta}
    =\mathcal{O}(1),\quad
    \sin2(\frac{\pi}{2}-\theta_{\delta})=2\cos\theta_{\delta}\sin\theta_{\delta}
    =\mathcal{O}(\delta^{\frac{p}{2(1-p)}})\qquad\text{as }\delta\to0.
  \end{equation}
  \hspace{13pt}Moreover, for the previously defined $u_{k,\delta}(\theta)=\delta\cos^{\frac{2}{p}-2k}\theta-
  \sin^{\frac{2}{p}-2k}\theta\ (k\in\mathbb{N})$, it can be expressed as 
  \begin{equation*}
    u_{k,\delta}'(\theta)=
    \begin{cases}
      (\frac{2}{p}-2k)(-\delta\sin\theta\cos^{\frac{2}{p}-(2k+1)}\theta-\cos\theta
      \sin\theta^{\frac{2}{p}-(2k+1)}\theta) & \text{if }\frac{2}{p}\geq2k+1,\\
      0 & \text{if }\frac{2}{p}=2k,
    \end{cases}
  \end{equation*}
  while for $\frac{2}{p}\geq2(k+1)$ it can be expressed as 
  \begin{align}
    u_{k,\delta}''(\theta)&=(\frac{2}{p}-2k)\Bigl(-\delta\cos^{\frac{2}{p}-2k}\theta+
    (\frac{2}{p}-(2k+1))\delta(1-\cos^{2}\theta)\cos^{\frac{2}{p}-2(k+1)}\theta\notag\\
    &\hspace{160pt}+\sin^{\frac{2}{p}-2k}\theta-(\frac{2}{p}-(2k+1))(1-\sin^{2}\theta)
    \sin^{\frac{2}{p}-2(k+1)}\theta\Bigr)\notag\\
    &=(\frac{2}{p}-2k)\Bigl(-u_{k}(\theta)+(\frac{2}{p}-(2k+1))(\delta\cos^{\frac{2}{p}
    -2(k+1)}\theta-\sin^{\frac{2}{p}-2(k+1)}\theta-(\delta\cos^{\frac{2}{p}-2k}\theta-
    \sin^{\frac{2}{p}-2k}\theta))\Bigr)\notag\\
    &=(\frac{2}{p}-2k)\Bigl(-(\frac{2}{p}-2k)u_{k,\delta}(\theta)+(\frac{2}{p}-2k-1)
    u_{k+1,\delta}(\theta)\Bigr),\label{u_rec}
  \end{align}
  so that for each $k$ the following holds.
  \begin{align}
  u_{k,\delta}(\frac{\pi}{2}-\theta_{\delta})&=\mathcal{O}(\delta)+\mathcal{O}
  (\delta^{\frac{p}{2(1-p)}\frac{2(1-kp)}{p}})\notag\\
  &=\mathcal{O}(\delta^{\frac{1-kp}{1-p}})
  \hspace{28pt}\text{as }\delta\to0,\text{ if }1\leq k,\text{ that is, }0\leq\frac{1-kp}{1-p}\leq 1, \label{u}\\
  u'_{k,\delta}(\frac{\pi}{2}-\theta_{\delta})&=-(\frac{2}{p}-2k)
  (\delta\cos\theta_{\delta}
  \sin^{\frac{2}{p}-(2k+1)}\theta_{\delta}-\sin\theta_{\delta}\cos^{\frac{2}{p}-(2k+1)}
  \theta_{\delta})\notag\\
  &=\mathcal{O}(\delta^{1+\frac{p}{2(1-p)}})+\mathcal{O}(\delta^{\frac{p}{2(1-p)}(
  \frac{2}{p}-(2k+1))})\notag\\
  &=\mathcal{O}(\delta^{\frac{2-p}{2(1-p)}})+\mathcal{O}(\delta^{\frac{2-(2k+1)p}{2(1-p)}
  })\notag\\
  &=\mathcal{O}(\delta^{\frac{2-(2k+1)p}{2(1-p)}})\quad\text{as }\delta\to0,\ \text{if }
  1\leq k\leq \frac{2-p}{2p},\text{ that is, }0\leq\frac{2-(2k+1)p}{2(1-p)}<1,\label{u'}\\
  u''_{k,\delta}(\frac{\pi}{2}-\theta_{\delta})&=\mathcal{O}(\delta^{\frac{1-pk}{1-p}})+
  \mathcal{O}(\delta^{\frac{1-(k+1)p}{1-p}})\notag\\
  &=\mathcal{O}(\delta^{\frac{1-(k+1)p}{1-p}})\hspace{13pt}\text{as }\delta\to0,\text{ if }
  1\leq k\leq \frac{1-p}{p},\text{ that is, }0\leq\frac{1-(k+1)p}{1-p}<1.\label{u''}
  \end{align}
  \hspace{13pt}On the other hand, recalling that the claim of Proposition \ref{Omega_1} is about the derivatives of $F_{p,\delta}$, we can express the desired $n(\leq\frac{2}{p})$-order derivatives from (\ref{F'}) and the Leibnitz rule as follows. Note that $a_{n,j},\ b_{n,j}\neq0$ are constant independent of $p,\ \delta$ and $\theta$.
  \begin{equation}\label{Lei}
    F_{p.\delta}^{(n)}(\theta)=
    \begin{cases}
      \displaystyle-\frac{1}{p}\sin2\theta\ u_{1,\delta}(\theta) & \text{if }n=1,\\
      \displaystyle-\frac{1}{p}\sum_{l=0}^{m-1}\Bigl(a_{n,2l}\cos2\theta\ u_{1,\delta}^{
      (2l)}(\theta)+a_{n,2l+1}\sin2\theta\ u_{1,\delta}^{(2l+1)}(\theta)\Bigr) & \text{if }n=2m,
      \\
      \displaystyle-\frac{1}{p}\sum_{l=0}^{m-1}\Bigl(b_{n,2l}\sin2\theta\ u_{1,\delta}^{
      (2l)}(\theta)+b_{n,2l+1}\cos2\theta\ u_{1,\delta}^{(2l+1)}(\theta)\Bigr)-
      \frac{1}{p}\sin2\theta\ u_{1,\delta}^{(2m)}(\theta) & \text{if }n=2m+1.
    \end{cases}
  \end{equation}
  Hence, the form of this expansion shows that we need to investigate the derivatives of $u_{1,\delta}$.
  \vspace{5pt}\par
  Firstly, in the cases $n=2m\ (m\geq2)$ or $n=2m+1$, for $0\leq l\leq m-1$, from (\ref{u_rec}) and (\ref{u}),
  \begin{align}
    u_{1,\delta}^{(2l)}(\frac{\pi}{2}-\theta_{\delta})=\mathcal{O}(
    u_{2,\delta}^{(2(l-1))}(\frac{\pi}{2}-\theta_{\delta}))
    &=\mathcal{O}(u_{3,\delta}^{(2(l-2))}(\frac{\pi}{2}-\theta_{\delta}))\notag\\
    &=\cdots\notag\\
    &=\mathcal{O}(u_{l+1,\delta}(\frac{\pi}{2}-\theta_{\delta}))\label{l+1}\\
    &=\mathcal{O}(\delta^{\frac{1-(l+1)p}{1-p}})\qquad\text{as }\delta\to0\notag
  \end{align}
  holds, and from this and (\ref{cos2}), in the cases $\frac{2}{p}=2m$ or $\frac{2}{p}=2m+1$, for $0\leq l\leq m-1$, we obtain the following evaluation formulas as $\delta\to0$.
  \begin{align}
    \cos2(\frac{\pi}{2}-\theta_{\delta})\ u_{1,\delta}^{(2l)}(\frac{\pi}{2}-
    \theta_{\delta})&=\mathcal{O}(\delta^{\frac{1-(l+1)p}{1-p}}),\label{u_2l_cos}\\
    \sin2(\frac{\pi}{2}-\theta_{\delta})\ u_{1,\delta}^{(2l)}(\frac{\pi}{2}-
    \theta_{\delta})&=\mathcal{O}(\delta^{\frac{p}{2(1-p)}+\frac{1-(l+1)p}{1-p}})
    =\mathcal{O}(\delta^{\frac{2-(2l+1)p}{2(1-p)}})\label{u_2l_sin}.
  \end{align}
  \par
  On the other hand, from (\ref{u'}), the evaluation formula
  \begin{equation}\label{u_rec_jk}
    u_{j,\delta}'(\frac{\pi}{2}-\theta_{\delta})=\mathcal{O}(u_{k,\delta}'(\frac{\pi}{2}-
    \theta_{\delta}))\qquad(j\leq k)\quad\text{ as }\delta\to0
  \end{equation}
  holds, and from this and (\ref{l+1}), in the cases $n=2m$ or $n=2m+1$ ($m\geq2$), for $0\leq l\leq m-2$, 
  \begin{align*}
    u_{1,\delta}^{(2l+1)}(\frac{\pi}{2}-\theta_{\delta})=\mathcal{O}(
    u_{l+1,\delta}'(\frac{\pi}{2}-\theta_{\delta}))
    &=\mathcal{O}(u_{m-1,\delta}'(\frac{\pi}{2}-\theta_{\delta}))\\
    &=\mathcal{O}\Bigl(\delta^{\frac{\frac{2}{p}-(2m-1)}{\frac{2}{p}(1-p)}}\Bigr)
    \qquad\text{as }\delta\to0
  \end{align*}
  holds. Thus, from this, the following two evaluation formulas are obtained.
  \begin{align}
    \cos2(\frac{\pi}{2}-\theta_{\delta})\ u_{1,\delta}^{(2l+1)}(\frac{\pi}{2}-
    \theta_{\delta})&=\mathcal{O}\Bigl(\delta^{\frac{\frac{2}{p}-(2m-1)}{\frac{2}{p}
    (1-p)}}\Bigr),\label{u_2l+1_cos}\\
    \sin2(\frac{\pi}{2}-\theta_{\delta})\ u_{1,\delta}^{(2l+1)}(\frac{\pi}{2}-
    \theta_{\delta})&=\mathcal{O}\Bigl(\delta^{\frac{1}{\frac{2}{p}(1-p)}+\frac{\frac{2}
    {p}-(2m-1)}{\frac{2}{p}(1-p)}}\Bigr)
    =\mathcal{O}\Bigl(\delta^{\frac{{\frac{2}{p}-2(m-1)}}{\frac{2}{p}(1-p)}}\Bigr)
    \qquad\text{as }\delta\to0\label{u_2l+1_sin}.
  \end{align}
  \par\vspace{7pt}
  Now, we only need to consider the remaining derivatives $u_{1,\delta}^{(2m-1)},\ u_{1,\delta}^{(2m)}$.\par\vspace{3pt}
  In the cases $n=2m\ (m\geq2)$ or $n=2m+1$, from (\ref{u_rec}) and (\ref{u_rec_jk}), the following holds.
  \begin{align}
    u_{1,\delta}^{(2m-1)}(\frac{\pi}{2}-\theta_{\delta})&=(\frac{2}{p}-2)(\frac{2}{p}-3)
    \cdots(\frac{2}{p}-(2m-1))u_{m,\delta}'(\frac{\pi}{2}-\theta_{\delta})+\mathcal{O}
    (u_{m-1,\delta}'(\frac{\pi}{2}-\theta_{\delta}))\notag\\
    &=(\frac{2}{p}-2)_{2(m-1)}
    u_{m,\delta}'(\frac{\pi}{2}-\theta_{\delta})+\mathcal{O}(\delta^{\frac{2-(2m-1)p}
    {2(1-p)}})\quad\text{as }\delta\to0.\label{u_2m-1}
  \end{align}\par
  Next, in the cases $\frac{2}{p}>2m+1$, by (\ref{l+1}) and (\ref{u''}),
  \begin{equation*}
    u_{1,\delta}^{(2m)}(\frac{\pi}{2}-\theta_{\delta})=\mathcal{O}(u_{m,\delta}''
    (\frac{\pi}{2}-\theta_{\delta}))=\mathcal{O}(\delta^{\frac{1-(m+1)p}{1-p}})
  \end{equation*}
  holds. Then, we note (\ref{cos2}) and (\ref{l+1}), and it is clear from the form of formulas for $u_{k,\delta}(\frac{\pi}{2}-
  \theta_{\delta})$ and $\sin2(\frac{\pi}{2}-\theta_{\delta})$ that the following evaluation holds, particularly with the addition of the $\Omega$ symbol.
  \begin{equation*}
    \sin2(\frac{\pi}{2}-\theta_{\delta})\ u_{1,\delta}^{(2m)}(\frac{\pi}{2}-
    \theta_{\delta})=\mathcal{O}(\delta^{\frac{p}{2(1-p)}+\frac{1-(m+1)p}{1-p}})
    =\mathcal{O},\ \Omega\Bigl(\delta^{\frac{\frac{2}{p}-(2m+1)}{\frac{2}{p}(1-p)}}\Bigr)
    \quad\text{as }\delta\to0.
  \end{equation*}
  \hspace{13pt}On the other hand, in the cases $\frac{2}{p}=2m+1$, since from
  \begin{align*}
    u_{m,\delta}(\theta)&=\delta\cos\theta-\sin\theta,\quad u_{m,\delta}''(\theta)=
    -u_{m,\delta}(\theta),\\
    u_{m,\delta}(\frac{\pi}{2}-\theta_{\delta})&=\delta\sin\theta_{\delta}-\cos
    \theta_{\delta}=\mathcal{O}(\delta)+\mathcal{O}(\delta^{\frac{p}{2(1-p)}})
    =\mathcal{O}(\delta^{\frac{p}{2(1-p)}})\qquad\text{as }\delta\to0,
  \end{align*}
  the evaluation formula
  \begin{equation*}
    \sin2(\frac{\pi}{2}-\theta_{\delta})u_{1,\delta}^{(2m)}(\frac{\pi}{2}-
    \theta_{\delta})
    =\mathcal{O}(\delta^{\frac{p}{2(1-p)}+\frac{p}{2(1-p)}})
    =\mathcal{O}(\delta^{\frac{p}{1-p}})\qquad\text{as }\delta\to0
  \end{equation*}
  holds, the above results can be summarized as follows.
  \begin{equation}\label{u_2m}
    \sin2(\frac{\pi}{2}-\theta_{\delta})u_{1,\delta}^{(2m)}(\frac{\pi}{2}-
    \theta_{\delta})=
    \begin{cases}
      \mathcal{O},\ \Omega\Bigl(\delta^{\frac{\frac{2}{p}-(2m+1)}{\frac{2}{p}(1-p)}}\Bigr)
       & \text{if }\frac{2}{p}>2m+1\\
      \mathcal{O}(\delta^{\frac{p}{1-p}}) & \text{if }\frac{2}{p}=2m+1
    \end{cases}
    \qquad\text{as }\delta\to0.
  \end{equation}
  \hspace{13pt}From the above, we are now ready to complete the proof of Proposition \ref{Omega_1}.
  \begin{proof}[\text{Proof of Proposition \ref{Omega_1}}]
    Let $\frac{2}{p}\in\mathbb{N}\setminus{\{1,2\}}$ and $1\leq n\leq\frac{2}{p}$.\par
  From the stationary point assumption (see Table 1), $F_{p,\delta}''(\theta)=-\frac{1}{p}(2\cos2\theta\ u_{1,\delta}(\theta)+\sin2\theta\ u_{1,\delta}'(\theta))$, (\ref{cos2}), (\ref{u}) and (\ref{u'}), 
    \begin{align*}
      F_{p,\delta}'(\frac{\pi}{2}-\theta_{\delta})&=0,\\
      F_{p,\delta}''(\frac{\pi}{2}-\theta_{\delta})&=\mathcal{O}(\delta)+\mathcal{O}
      (\delta^{\frac{p}{2(1-p)}+\frac{2-3p}{2(1-p)}})
      =\mathcal{O}(\delta)\Bigl(=\mathcal{O}\Bigl(\delta^{\frac{\frac{2}{p}-2}
      {\frac{2}{p}(1-p)}}\Bigr)\Bigr)\qquad\text{as }\delta\to0
    \end{align*}
    are obvious, so we consider the cases $3\leq n\leq\frac{2}{p}$ from now on.
    \vspace{10pt}\par

    (1) In the cases $n=2m\ (m\geq2)$, from (\ref{Lei}), (\ref{u_2l_cos}) and  (\ref{u_2l+1_sin}), the following holds.
    \begin{align*}
      F_{p,\delta}^{(2m)}(\frac{\pi}{2}-\theta_{\delta})=-\frac{a_{n,2(m-1)}}
      {p}&\cos2(\frac{\pi}{2}-\theta_{\delta})\ u_{1,\delta}^{(2(m-1))}(\frac{\pi}{2}-
      \theta_{\delta})\\
      &-\frac{a_{n,2m-1}}{p}\sin2(\frac{\pi}{2}-\theta_{\delta})\ u_{1,\delta}^{(2m-1)}
      (\frac{\pi}{2}-\theta_{\delta})+\mathcal{O}\Bigl(\delta^{\frac{\frac{2}{p}-2(m-1)}
      {\frac{2}{p}(1-p)}}\Bigr)\qquad\text{as }\delta\to0.
    \end{align*}
    Then, noting (\ref{l+1}), the first term on the right-hand side can be evaluated as 
    \begin{equation*}
      -\frac{a_{n,2(m-1)}}{p}\cos2(\frac{\pi}{2}-\theta_{\delta})\ 
      u_{1,\delta}^{(2(m-1))}(\frac{\pi}{2}-\theta_{\delta})=\mathcal{O},\ \Omega\Bigl(
      \delta^{\frac{\frac{2}{p}-2m}{\frac{2}{p}(1-p)}}\Bigr)\qquad\text{as }\delta\to0
    \end{equation*}
    (by the $\Omega$ symbol particularly) from (\ref{u_2l_cos}) and the form of the expressions of $u_{k,\delta}
    (\frac{\pi}{2}-\theta_{\delta})$ and $\cos2(\frac{\pi}{2}-\theta_{\delta})$, and for the second term, from (\ref{cos2}) and (\ref{u_2m-1}), the following holds.
    \begin{equation*}
      \sin2(\frac{\pi}{2}-\theta_{\delta})\ u_{1,\delta}^{(2m-1)}(\frac{\pi}{2}-
      \theta_{\delta})=(\frac{2}{p}-2)_{2(m-1)}\sin2(\frac{\pi}{2}-\theta_{\delta})\ 
      u_{m,\delta}'(\frac{\pi}{2}-\theta_{\delta})+\mathcal{O}
      \Bigl(\delta^{\frac{\frac{2}{p}-2(m-1)}{\frac{2}{p}(1-p)}}\Bigr)\quad\text{as }\delta\to0,
    \end{equation*}
    \hspace{13pt}In addition, combining (\ref{cos2}) and (\ref{u'}) again, that is,
    \begin{equation*}
      \sin2(\frac{\pi}{2}-\theta_{\delta})\ u_{m,\delta}'(\frac{\pi}{2}-\theta_{\delta})=
      \begin{cases}
        \mathcal{O}\Bigl(\delta^{\frac{\frac{2}{p}-2m}{\frac{2}{p}(1-p)}}\Bigr) & 
        \text{if }2m\leq\frac{2}{p}-1\\
        0 & \text{if }2m=\frac{2}{p}\\
      \end{cases}
      \qquad\text{as }\delta\to0,
    \end{equation*}
    then we obtain
    \begin{equation*}
      F_{p,\delta}^{(2m)}(\frac{\pi}{2}-\theta_{\delta})=
      \begin{cases}
        \mathcal{O}\Bigl(\delta^{\frac{\frac{2}{p}-2m}{\frac{2}{p}(1-p)}}\Bigr) & 
        \text{if }2m\leq\frac{2}{p}-1\\
        \mathcal{O},\ \Omega(1) & \text{if }2m=\frac{2}{p}
      \end{cases}
      \qquad\text{as }\delta\to0.
    \end{equation*}
    \vspace{5pt}\par
    (2) In the cases $n=2m+1\ (m\in\mathbb{N})$, from (\ref{Lei}), (\ref{u_2l_sin}) and 
    (\ref{u_2l+1_cos}), the following holds.
    \begin{align*}
      F_{p,\delta}^{(2m+1)}(\frac{\pi}{2}-\theta_{\delta})=-\frac{b_{n,2m-1}}{p}&
      \cos2(\frac{\pi}{2}-\theta_{\delta})\ u_{1,\delta}^{(2m-1)}(\frac{\pi}{2}-
      \theta_{\delta})\\
      &-\frac{b_{n,2m}}{p}\sin2(\frac{\pi}{2}-\theta_{\delta})\ u_{1,\delta}^{(2m)}
      (\frac{\pi}{2}-\theta_{\delta})+\mathcal{O}\Bigl(\delta^{\frac{\frac{2}{p}-(2m-1)}
      {\frac{2}{p}(1-p)}}\Bigr)\qquad\text{as }\delta\to0.
    \end{align*}
    Then, noting the form of each expression, the first term on the right-hand side can be evaluated as
    \begin{align*}
    \cos2(\frac{\pi}{2}-\theta_{\delta})\ u_{1,\delta}^{(2m-1)}(\frac{\pi}{2}-
    \theta_{\delta})&=(\frac{2}{p}-2)_{2(m-1)}\cos2(\frac{\pi}{2}-\theta_{\delta})\ 
    u_{m,\delta}'(\frac{\pi}{2}-\theta_{\delta})+\mathcal{O}\Bigl(\delta^{\frac{\frac{2}
    {p}-(2m-1)}{\frac{2}{p}(1-p)}}\Bigr)\\
    &=\mathcal{O},\ \Omega\Bigl(\delta^{\frac{\frac{2}{p}-(2m+1)}{\frac{2}{p}(1-p)}}\Bigr)+\mathcal{O}
    \Bigl(\delta^{\frac{\frac{2}{p}-(2m-1)}{\frac{2}{p}(1-p)}}\Bigr)\qquad\text{as }\delta\to0
    \end{align*}
    (by the $\Omega$ symbol particularly) from (\ref{cos2}), (\ref{u'}) and (\ref{u_2m-1}). \par
    Therefore, with the above and (\ref{u_2m}), we conclude.
    \begin{equation*}
       F_{p,\delta}^{(2m+1)}(\frac{\pi}{2}-\theta_{\delta})=
      \begin{cases}
        \mathcal{O},\ \Omega\Bigl(\delta^{\frac{\frac{2}{p}-(2m+1)}{\frac{2}{p}(1-p)}}\Bigr) & 
        \text{if }2m+1\leq\frac{2}{p}-1\\
        \mathcal{O},\ \Omega(1) & \text{if }2m+1=\frac{2}{p}
      \end{cases}
      \qquad\text{as }\delta\to0.
    \end{equation*}
  \end{proof}

  \section{Concluding remarks}
    \hspace{13pt}In order to achieve our future goal \textit{generalizing Proposition \ref{prop2022} (\cite{Kuratsubo-2022}, Lemma 5.1) for $p$ via Theorem \ref{thm} (series representations of $D_{\beta}^{[p]}-\mathcal{D}_{\beta}^{[p]}$ consisting of $J_{\beta+1}^{[p]}$)}, we need to identify the infimum of $\beta$ such that $\mathcal{D}_{\beta}^{[p]}(1:x)$ is integrable on $\mathbb{R}^{2}$, the assumption of Theorem \ref{thm}. In particular, considering the case of $\frac{2}{p}\in\mathbb{N}\setminus{\{1,2\}}$, if we can find infimum $Q^{[p]}$ of $q_{\omega}^{[p]}>0$ (particularly $q_{0}^{[p]}=\frac{p}{2}$) which satisfy uniform evaluation formulas on $\mathbb{R}^{2}$
  \begin{equation}\label{q-order}
    J_{\omega}^{[p]}(x)\stackrel{\text{unif}}{=}\mathcal{O}(|x|_{p}^{-q^{[p]}_{\omega}})\qquad\text{as }|x|_{p}\to\infty,\quad\text{for }\omega\geq0,
  \end{equation}
  then we also obtain the desired $\beta$. Note that, unlike $J_{\alpha}(r)=\mathcal{O}(r^{-\frac{1}{2}})$, we generally consider that the exponents depend on the orders.\par
  In fact, under $\beta>1-Q^{[p]}$, since there exists $\varepsilon>0$ satisfying $\beta+1>2-Q^{[p]}+\varepsilon$, the following holds by $|x|_{p}^{-(\beta+1)}<|x|_{p}^{-(2-Q^{[p]}+\varepsilon)}$.
  \begin{equation}\label{frac-est}
    \frac{J_{\beta+1}^{[p]}(x)}{|x|_{p}^{\beta+1}}=\mathcal{O}(|x|_{p}^{-q_{\beta+1}^{[p]}-(2-Q^{[p]}+\varepsilon)})=\mathcal{O}(|x|_{p}^{-(2+\varepsilon)})=\mathcal{O}(|x|^{-(2+\varepsilon)}).
  \end{equation}\par
  Also, from Proposition 2.6 of \cite{K1} (series representation of $J_{\omega}^{[p]}$), for $x\neq0$, 
  \begin{align*}
    \frac{J_{\omega}^{[p]}(x)}{|x|_{p}^{\omega}}&=\frac{4}{p^{\omega+2}\Gamma^{2}(\frac{1}{p})}\left(\frac{\Gamma^{2}(\frac{1}{p})}{\Gamma(\frac{2}{p}+\omega)}+\sum_{k=1}^{\infty}\frac{(-1)^{k}}{\Gamma(\frac{2(k+1)}{p}+\omega)}\sum_{m\in\mathbb{N}_{0}^{2}\ |m|'=k}\frac{\Gamma(\frac{2m+1}{p})}{(2m)!}x^{2m}\right)\\
    &\to\frac{4}{p^{\omega+2}\Gamma^{2}(\frac{1}{p})}\left(\frac{\Gamma^{2}(\frac{1}{p})}{\Gamma(\frac{2}{p}+\omega)}+0\right)\qquad\text{as }x\to0
  \end{align*}
  holds (note the notation of multiple exponents), and $\frac{J_{\omega}^{[p]}(x)}{|x|_{p}^{\omega}}$ is continuous on $\mathbb{R}^{2}$ if we redefine the value of the function at the origin as follows.
  \begin{equation}
    \frac{J_{\omega}^{[p]}(x)}{|x|_{p}^{\omega}}:=\frac{(\frac{2}{p})^{2}}{p^{\omega}\Gamma(\omega+\frac{2}{p})}\qquad\text{for }x=0.
  \end{equation}\par
  Hence, by the uniformly asymptotic evaluation formulas (\ref{frac-est}) and continuity of $\frac{J_{\omega}^{[p]}(x)}{|x|_{p}^{\omega}}$ on $\mathbb{R}^{2}$, there exist a closed sphere $B$ of origin center and constants $C_{B},\ C'_{B}>0$, and the following holds.
  \begin{equation*}
    \int_{\mathbb{R}^{2}}\left|\frac{J_{\omega}^{[p]}(x)}{|x|_{p}^{\omega}}\right|dx
    =\int_{B}\left|\frac{J_{\omega}^{[p]}(x)}{|x|_{p}^{\omega}}\right|dx+\int_{\mathbb{R}^{2}\setminus B}\left|\frac{J_{\omega}^{[p]}(x)}{|x|_{p}^{\omega}}\right|dx\leq C_{B}+C'_{B}\int_{\mathbb{R}^{2}\setminus B}\frac{dx}{|x|^{2+\varepsilon}}<+\infty.
  \end{equation*}\par\vspace{10pt}
  From the above, $\mathcal{D}_{\beta}^{[p]}(1:x)\ (\backsimeq_{\beta,p}\frac{J_{\beta+1}^{[p]}(2\pi x)}{|2\pi x|_{p}^{\beta+1}};\ \cite{K1},\text{ Proposition 3.1})$ is integrable on $\mathbb{R}^{2}$, that is, the infimum of $\beta$ satisfying the assumption of Theorem \ref{thm} is identified as $1-Q^{[p]}$. Then, it can be written as a special case of Theorem \ref{thm} as follows. Thus, by this theorem, tackling a part of the problem is attributed to investigating the uniformly asymptotic evaluations of the generalized Bessel functions $J_{\omega}^{[p]}$.  
  \begin{theorem}[\itshape{\textit{Astroid-type} versions of Theorem \ref{thm}}]\label{thm_At}
  \itshape{Let $p>0$ such that $\frac{2}{p}\in\mathbb{N}\setminus{\{1,2\}}$. If there exists $Q^{[p]}:=\inf_{\omega\geq0}q_{\omega}^{[p]}$ satisfying $J_{\omega}^{[p]}(x)\stackrel{\text{unif}}{=}\mathcal{O}(|x|_{p}^{-q_{\omega}^{[p]}})$, then the following holds for $\beta>1-Q^{[p]}$.
  \begin{equation*}
    D_{\beta}^{[p]}(s:x)-\mathcal{D}_{\beta}^{[p]}(s:x)=s^{\beta+\frac{2}{p}}p^{\beta+1}
    \Gamma^{2}(\frac{1}{p})\sum_{n\in\mathbb{Z}^{2}\setminus\{0\}}\frac{J_{\beta+1}^{[p]}
    (2\pi\sqrt[p]{s}(x-n))}{(2\pi \sqrt[p]{s}|x-n|_{p})^{\beta+1}}
    \qquad\text{for }s>0,\ x\in\mathbb{R}^{2}.\vspace{-5pt}
  \end{equation*}}
  Furthermore, under this assumption, the series converges absolutely for $x\in\mathbb{T}^{2}$.
  \end{theorem}
  \vspace{7pt}
  Therefore, as the next step of our research, we aim to obtain the uniformly asymptotic evaluations (\ref{q-order}) and the infimum $Q^{[p]}$, but for our future purpose mentioned above, we infer that the evaluations for the cases $\omega\in\mathbb{N}$, at least for $\omega=2$, is sufficient. \par
  Thus, we would like to derive the corresponding representations for general $p$ (oscillatory integral representations of $J_{n}^{[p]}$) to the ones of the Bessel functions(\cite{Watson}, p19-p20) 
  \begin{equation}\label{OIRBF}
    J_{n}(r)=\frac{1}{2\pi}\int_{0}^{2\pi}e^{i(r\sin\theta-n\theta)}d\theta\qquad\text{for }n\in\mathbb{N}.
  \end{equation}\par
  It is important to note that the order $n$ is only included in the phase function, and if this derivation succeeds, then the uniformly asymptotic evaluations (\ref{q-order}), for the cases when $\omega$ are natural numbers, and $Q^{[p]}_{\mathbb{N}}:=\inf_{n\in\mathbb{N}}q_{n}^{[p]}$ are immediately obtained by applying Van der Corput's lemma (Lemma \ref{-1/k_order}).
  \par\vspace{5pt}
  The derivation of the oscillatory integral representations (\ref{OIRBF}) based on the Laurent expansion 
  \begin{equation*}
  e^{\frac{1}{2}r(z-\frac{1}{z})}=\sum_{m=-\infty}^{\infty}J_{m}(r)z^{m}\qquad\text{for }r>0,\ z\in\mathbb{C}\setminus\{0\}
  \end{equation*}
  is well known. Therefore, we follow this method, and from the oscillatory integral representation (\ref{odd-osc}) under the restricted condition such that $\frac{2}{p}\in\mathbb{N}$ is odd, infer that the following equality holds, which we will endeavor to prove it in next studies. Note that we define $\frac{2}{p}=:2j+1,\ \psi_{j}(z):=(-\frac{1}{16})^{j}(z^{2}-(\frac{1}{z})^{2})^{2j}$.
  \begin{equation}
    e^{\eta_{1}(-1)^{j}(\frac{1}{2}(z-\frac{1}{z}))^{\frac{2}{p}}+i\ \eta_{2}(\frac{1}{2}(z+\frac{1}{z}))^{\frac{2}{p}}}\psi_{j}(z)\stackrel{\text{??}}{=}\sum_{m=-\infty}^{\infty}J_{m}^{[p]}(\eta)z^{m}\qquad\text{for }\eta\in\mathbb{R}^{2},\ z\in\mathbb{C}\setminus\{0\}.
  \end{equation}

  \section*{Acknowledgement}
  \hspace{13pt}I would like to express my gratitude to Prof. Mitsuru Sugimoto and Prof. Kohji Matsumoto for numerous constructive suggestions and helpful remarks on harmonic analysis and number theory. \par
The author is financially supported by JST SPRING, Grant Number JPMJSP2125, and would like to take this opportunity to thank the ``THERS Make New Standards Program for the Next Generation Researchers.'' 
   
  \itshape{
  \hspace{13pt}The author's affiliation: Graduate School of Mathematics, Nagoya University, Chikusa-ku, Nagoya 464-8602, Japan\\
  \hspace{13pt}The author's email address: kitajima.masaya.z5@s.mail.nagoya-u.ac.jp}
\end{document}